\documentclass[12pt]{article}
\usepackage{amsmath}
\usepackage{amsfonts}
\usepackage{amssymb}
\usepackage{graphicx}
\usepackage{amsthm}
\usepackage{enumerate}
\usepackage{url,xargs}
\usepackage{hyperref}

\begin{document}

\title{Flows and bisections in cubic graphs}

\author{L. Esperet \footnote{Laboratoire G-SCOP (Univ. Grenoble Alpes, CNRS),
    Grenoble, France. Partially supported by ANR Project STINT
  (\textsc{anr-13-bs02-0007}) and LabEx PERSYVAL-Lab
  (\textsc{anr-11-labx-0025}).} \and
  G. Mazzuoccolo \footnote{Department of Computer Science, University of Verona, Italy. E-mail: {\tt giuseppe.mazzuoccolo@univr.it}}
    \and M. Tarsi \footnote{The Blavatnik School of Computer Science,
    Tel Aviv University, Israel. E-mail: {\tt tarsi@post.tau.ac.il}}}

\maketitle

\newtheorem{theorem}{Theorem}
\newtheorem{lemma}[theorem]{Lemma}
\newtheorem{claim}[theorem]{Claim}
\newtheorem{definition}[theorem]{Definition}
\newtheorem{example}[theorem]{Example}
\newtheorem{corollary}[theorem]{Corollary}
\newtheorem{conjecture}[theorem]{Conjecture}
\newtheorem{observation}[theorem]{Observation}
\newtheorem{proposition}[theorem]{Proposition}

\newcommand{\Po}{\cal{P}^*}
\newcommand{\Pa}{\cal{P}}
\newcommand{\Qo}{\cal{Q}^*}
\newcommand{\Qa}{\cal{Q}}
\def\scc{\mathop{\mathrm{scc}}\nolimits}

\begin{abstract} \noindent
A \emph{$k$-weak bisection} of a cubic graph $G$ is a partition of the
vertex-set of $G$ into two parts
$V_1$ and $V_2$ of equal size, such that each connected component of
the subgraph of $G$ induced by $V_i$ ($i=1,2$) is a tree of at most
$k-2$ vertices. This notion can be viewed as a relaxed version of
nowhere-zero flows, as it directly follows from old results of Jaeger that
every cubic graph $G$ with a circular nowhere-zero $r$-flow has a
\emph{$\lfloor r \rfloor$-weak bisection}. In this paper we study problems related to the existence of $k$-weak
bisections. We believe that every cubic graph which has a
perfect matching, other than the Petersen graph, admits a 4-weak
bisection and we present a family of cubic graphs with no perfect
matching which do not admit such a bisection. The main result of
this article is that every cubic graph admits a 5-weak bisection.
When restricted to bridgeless graphs, that result would be a
consequence of the assertion of the 5-flow Conjecture
and as such it can be considered a (very small) step toward
proving that assertion. However, the harder part of our proof
focuses on graphs which do contain bridges.
\end{abstract}


\section{Strong and weak bisections and how they relate to nowhere-zero
flows}\label{sec:flo}

This introductory section is mostly a brief reformulation and
adaptation of some results stated and proved in \cite{bvj} in
terms of integer nowhere-zero flows and then in \cite{steff} in
the more general setting of real valued flows.

\smallskip

\begin{definition}
 Given a real number $r\geq 2$, a {\bf circular nowhere-zero
$r$-flow} ($r$-{\sc cnzf} for short) in a graph $G=(V,E)$ is an
assignment $f:E \rightarrow [1,r-1]$ and an orientation $D$ of
$G$, such that $f$ is a {\bf flow} in $D$. That is, for every
vertex $x\in V$, $\sum_{e\in E^+(x)}f(e)=\sum_{e\in E^-(x)}f(e)$,
where $E^+(x)$, respectively $E^-(x)$, are the sets of edges
directed from, respectively towards $x$ in $D$.
\end{definition}
Accordingly defined is:
\begin{definition} The \textbf{circular flow number} $\phi_c(G)$ of a graph $G$ is the
infimum of the set of real numbers $r$, for which $G$ admits an
$r$-{\sc cnzf}. If $G$ has a bridge then we define
$\phi_c(G)=\infty$.
\end{definition}

The notion of $r$-{\sc cnzf} was first introduced in \cite{gtz},
while observing that $(k,d)$-coloring, previously studied by Bondy
and Hell \cite{BH}, can be interpreted as the dual of a real
(rather than integer)-valued nowhere-zero flow. The existence of
$\phi_c(G)$ and the fact that it is always rational, was also
observed and stated in  \cite{gtz}. \textbf{Integer} nowhere-zero
flows are much more widely known and intensively studied, since
first presented by W. Tutte \cite{5flow}, 60 years ago.

\bigskip

Along the article we use the following notational conventions:
 Let $G=(V,E)$ be a graph.

\begin{itemize}

 \item
 Given a set of vertices $X \subseteq V$, the set of edges with one
 endvertex in $X$ and the other one in $V \setminus X$, known as the
 \textbf{edge-cut} induced by $X$, is denoted here by $E(X)$.

 \item A \textbf{bridge} is an edge-cut which consists of a single edge. A
 graph with no bridge is \textbf{bridgeless}.

 \item When $D$ is an orientation of $G$, $E(X)$ is partitioned
 into $E^+(X)$ consisting of the edges directed from
 $X$ to $V\setminus X$, and $E^-(X)$, the set of edges directed
 from $V\setminus X$ to $X$.

\end{itemize}

\begin{definition}
A {\bf bisection} of a cubic graph $G=(V,E)$ is a partition of its
vertex set $V$ into two disjoint subsets $V_1$ and $V_2$ of the
same cardinality $|V_1|=|V_2|$.
\end{definition}

For reasons clarified in the sequel we also define:

\begin{definition}
Let $k\ge 3$ be an integer. A bisection $(V_1,V_2)$ of a cubic
graph $G=(V,E)$ is a \textbf{$k$-weak bisection} if every
connected component of each of the two subgraphs of $G$, induced by
$V_1$ and by $V_2$  is a tree on at most $k-2$ vertices. Such a
component is referred to in the sequel as \textbf{monochromatic}.
\end{definition}

\medskip

Let $D$ be an orientation of a graph $G=(V,E)$ and let $f$ be an
$r$-{\sc cnzf} in $D$. Then for any $x \in V$, both $E^+(\{x\})$,
and $E^-(\{x\})$ are non-empty. If $G$ is cubic then $V$ is
partitioned into two sets $V_1,V_2$, where $V_i$ consists of all
vertices of out-degree $i$ in $D$. That partition is clearly a
bisection.

Furthermore, assume that there is a monochromatic path $xyz$ in
$V_1$. Counting the relevant ingoing and outgoing edges reveals
$E^+(\{x,y,z\})=1$ and $E^-(\{x,y,z\})=4$, and then it follows
from the definition of an $r$-{\sc cnzf}  that $r\geq (4/1)+1=5$.
If $V_1$ contains a monochromatic cycle, then all the edges with
exactly one endpoint in this cycle are directed toward the cycle
in  $D$, which contradicts the fact that $D$ has an $r$-{\sc cnzf}
(regardless of the value of $r$). Clearly $V_1$ can be replaced
here by $V_2$.
 Hence, we have shown that any cubic graph $G$ with $\phi_c(G)<5$ has a 4-weak bisection.
The same degree counting arguments leads to the following general
observation, explicitly stated and proved (for integer values of
$r$) in \cite{bvj} :

\begin{proposition}
\label{local} If a cubic graph $G$ admits an $r$-{\sc cnzf} then
there exists a $\lfloor r\rfloor$-weak bisection of $G$ ($\lfloor
r\rfloor$ is the largest integer, smaller or equal to $r$).
\end{proposition}

Also stated and proved in \cite{bvj} and \cite{steff} is the
following stronger theorem, which fully characterizes cubic graphs
(an instance of a more general statement which refers to all
graphs) which admit $r$-{\sc cnzf}s, in terms of bisections:

\begin{theorem}
\label{rb} Let $r \geq 3$ be a real number. A cubic graph
$G=(V,E)$ admits an $r$-{\sc cnzf} if and only if there exists a
bisection $(V_1,V_2)$ of $G$, such that for every set of vertices
$X\subseteq V$, \[|E(X)|\geq \tfrac r{r-2}\,||V_1 \cap X| -|V_2
\cap X|| \]
\end{theorem}

The "only if" part of Theorem~\ref{rb} is proved by applying the
same counting argument as above to all subsets $X$ of $V$, while
the proof of Proposition~\ref{local} only considers the subsets
$X$ of $V_1$ and $V_2$. The "if" part also requires some basic
flow network theory.

We refer to a bisection which satisfies the condition of Theorem
\ref{rb} as an \textbf{$r$-strong bisection} of $G$. Note that an
$r$-strong bisection is clearly also a $\lfloor r\rfloor$-weak
bisection.

By definition, a cubic graph $G$ admits a $3$-weak bisection if
and only if it is bipartite, in which case that bisection
$(V_1,V_2)$ consists of the two "sides" of the bipartition of $V$.
It is easy to observe that $(V_1,V_2)$ in that case is also a
$3$-strong bisection. Proposition \ref{local} hence implies that
the circular flow number $\phi_c(G)$ of a cubic graph $G$ is never
strictly between 3 and 4.

\section{On 4-weak bisections} \label{contbl}

It had been suggested by Amir Ban and Nati Linial \cite{BanLin}
that every cubic graph other than the Petersen graph, admits a
4-weak bisection (where every monochromatic component is either an
isolated vertex or a single edge). However, we later found an
infinite family of counterexamples to that statement. These
examples, hereby presented, are also quoted in \cite{BanLin2}.

 A frequently used module in our construction is the
 graph $L_k$ ($k\ge 0$) depicted
in Figure~\ref{fig:Lk}. Note that $L_0$ is obtained from $K_{3,3}$
by the removal of one edge.

\begin{figure}[htb]
\centering \includegraphics[scale=0.8]{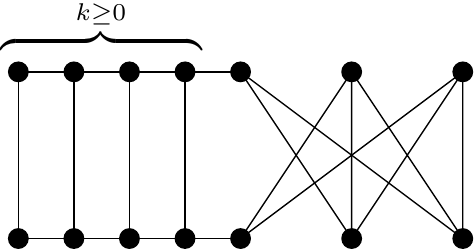} \caption{ $L_k$}
\label{fig:Lk}
\end{figure}

We say that a 2-coloring of the vertices of a graph (not
necessarily cubic) is \textbf{balanced} if the two color classes
are of the same size.

\begin{lemma}\label{lem:Lk}
For any $k\ge 0$, any $2$-coloring of $L_k$ with no monochromatic
component on more than 2 vertices is balanced and the two vertices
of degree 2 are of distinct colors.
\end{lemma}
\begin{proof}
We prove the result by induction on $k$. The two vertices of
degree 2 are denoted by $u$ and $v$. When referring to the color
of a vertex we relate to any given $2$-coloring of $L_k$ with no
monochromatic component of size 3 or more.

Assume first that $k=0$. If two adjacent vertices have the same
color, say 1, then all four remaining vertices have to be colored
2 and some vertex colored 2 has two neighbors colored 2 as well,
which is a contradiction. It follows that without loss of
generality the two neighbors of $u$ (resp. $v$) are colored 1
(resp. 2). As a consequence, $u$ and $v$ are colored 1 and 2,
respectively and the coloring is indeed balanced.

If $k\ge 1$, then by removing $u$ and $v$ we obtain $L_{k-1}$. By
the induction hypothesis, the coloring of the copy of $L_{k-1}$ is
balanced and the two vertices of degree 2 of the copy of $L_{k-1}$
have distinct colors. It follows that if $u$ and $v$ have the same
color, then one of them has the same color as its two neighbors,
which is a contradiction. Therefore, $u$ and $v$ have distinct
colors, and the coloring is balanced.
\end{proof}

For any $k\ge 0$, let $L'_k$ be the graph obtained from $L_k$ by
adding a vertex adjacent to the two vertices of degree 2 of $L_k$,
and let $T_k$ be the graph obtained by taking one copy of $L'_k$
and two copies of $L'_0$ and adding a new vertex adjacent to the
three vertices of degree 2 (see Figure~\ref{fig:counterexample}).

\begin{proposition}\label{p:counter}
For any $k\ge 0$, the graph $T_k$ has no 4-weak bisection.
\end{proposition}
\begin{proof}
Let $x$ be the vertex whose neighbors are the vertices of degree
two in the copies of $L'_k$ and $L'_0$. Consider an arbitrary
4-weak bisection of $T_k$. Note that by Lemma~\ref{lem:Lk}, the
corresponding 2-colorings of the copies of $L_k$ and $L_0$ are
balanced, and the color of $x$ has to be distinct from the colors
of its three neighbors. Therefore, the sizes of the two color
classes differ by 2.
\end{proof}

The two copies of $L_0$ can clearly be replaced by one copy of
$L_i$ and one of  $L_j$, for any  $i,j\ge 0$.

\bigskip

\begin{figure}[htb]
\centering
\includegraphics[width=7cm]{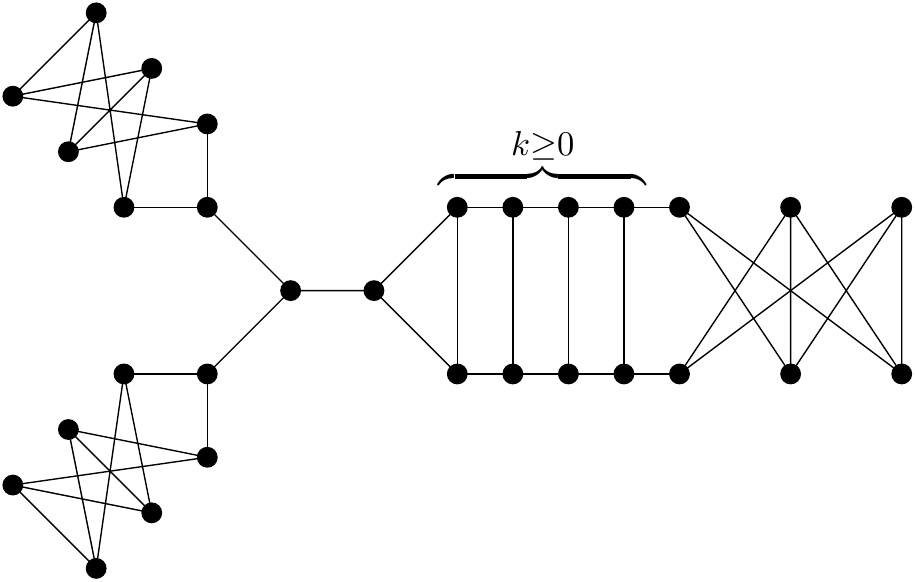}
\caption{$T_k$, an infinite family of cubic graph with no 4-weak
bisections}  \label{fig:counterexample}
\end{figure}

It so happens that each of our counterexamples contains a
\textbf{bridge}. Furthermore, all these examples lack a perfect
matching, and, as far as we can tell, that may be true for every
counterexample. We therefore rephrase Ban and Linial's idea as
follows:
\bigskip
\begin{conjecture}\label{conj:mod-banlin}
The Petersen graph is the only cubic graph which
admits a perfect
  matching, but no $4$-weak bisection.
\end{conjecture}

A slightly weaker version of Conjecture \ref{conj:mod-banlin},
referring to {\bf bridgeless} cubic graphs (which do admit perfect
matchings \cite{peter}) is included in \cite{BanLin2}.

As a result of Proposition \ref{local}, every cubic graph $G$ with
$4\leq \phi_c(G)< 5$ admits a $4$-weak bisection. Accordingly, a
counterexample for Conjecture \ref{conj:mod-banlin} should be
looked for among cubic graphs $G$ with  $\phi_c(G)\geq 5$. When
restricted to bridgeless graphs, there are several known simple
schemas to recursively construct infinitely many cubic graphs, 2-
and 3-connected, with circular flow number 5, from smaller ones.
We checked several such graphs, obtained from the Petersen graph, and
found them to admit $4$-weak bisections, but we can offer nothing
in the direction of a general proof. We recently \cite{emt}
developed an arsenal of construction methods which yield a rich
variety of  {\bf snarks} (cyclically 4-edge-connected cubic graphs
of girth at least 5 which admit no 4-{\sc cnzf}) $G$ with $\phi_c(G)\geq
5$. Conjecture \ref{conj:mod-banlin}, as well as its restriction
to bridgeless cubic graphs, seems, at least for now, far beyond
our reach.

\smallskip

The vertices of any cubic graph can be partitioned into two color
classes with no monochromatic component on more than two vertices, by
repeatedly switching the color of a vertex which shares its color
with more than one of its three neighbors. The resulting partition is not always a bisection, as $|V_1|=|V_2|$ is not
guaranteed. However, the Petersen graph as well as $T_k$ for every
$k \geq 0$ admit such a partition where $|V_1|-|V_2|=2$. This
observation leads to the following:

\begin{conjecture}\label{conj:con3}
The vertex set of any cubic graph can be partitioned into two
color classes $V_1$ and $V_2$, satisfying $||V_1|-|V_2||\leq 2$, such that
each monochromatic component contains at most two vertices.
\end{conjecture}

\medskip

Unlike nowhere-zero flows and $r$-strong bisections, the existence
of $k$-weak bisections is not restricted to bridgeless graphs.
Hence, the existence of a $6$-weak bisection in every cubic graph
cannot be directly deduced from Seymour's 6-flow Theorem
\cite{seym6}. A stronger result however holds, as we  prove in the
following section.

\section{5-weak bisections}\label{s:5}
As noted in the previous section, a $k$-weak bisection can be
considered as a relaxed local variant of  $k$-{\sc nzf}, at least
when referring to $2$-connected cubic graphs.  When comparing our
results to what is known about strong bisections ($k$-{\sc
nzf}'s), we get the flavor of "One step down": Tutte's $5$-flow
conjecture asserts the existence of a $5$-strong bisection for
every $2$-connected cubic graph. In
Conjecture~\ref{conj:mod-banlin} (with the exception of the Petersen graph), $5$-strong bisections are replaced by $4$-weak
bisections. With the following theorem, not only we replace
"$6$-strong bisections" in Seymour's $6$-flow theorem (when
stated, without any loss of generality, for cubic graphs), by
"$5$-weak bisections", but we also manage to do so for all cubic
graphs, not necessarily bridgeless. Theorem \ref{th:main} can be
viewed as a \textbf{local} $5$-flow theorem (although its
restriction to bridgeless graphs happens to be the easier case).

\begin{theorem}\label{th:main}
Every cubic graph admits a 5-weak bisection.
\end{theorem}
 \medskip

Our proof of Theorem \ref{th:main} goes through a sequence of
preparatory technical claims.

\begin{definition}

A \textbf{valid factor} $\Pa$ of a graph $G$ is a spanning
subgraph of $G$, where every connected component is either a
cycle, or a path, with the property that an odd (number of
vertices) path of $\Pa$ is not an isolated vertex and its two
endpoints are non-adjacent in $G$.

\end{definition}

We first observe that a valid factor indeed exists in every cubic
graph $G$: By Vizing's theorem, $G$ has a (proper)
4-edge-coloring. Select all the edges of $G$ colored $1$ or $2$.
The subgraph obtained is a spanning subgraph whose connected
components are paths and (even) cycles, and none of them is an
isolated vertex. Now insert every edge of $G$ which connects the
two endpoints of an odd path of the factor, to turn it into a
cycle. The spanning subgraph obtained that way is a valid factor
of $G$.

Let $\Po$ be a valid factor of $G$ such that the number of
components of $\Po$ is minimal, and, subject to that condition,
the number of odd cycles of $\Po$ is minimal.

\begin{claim}\label{cl:edgeendpoint}
We say that a vertex of $G$ is \textbf{external}, if it is the
endpoint of a path of $\Po$, or it belongs to an odd cycle of
$\Po$.

There is no edge in $G$ which connects two external vertices on
two distinct components of $\Po$.
\end{claim}

\begin{proof}
Let $v_1v_2$ be an edge of $G$, where $v_1$ and $v_2$ are external and
belong, respectively, to $T_1$ and $T_2$, two distinct components of
$\Po$. The number of components of $\Po$ is decreased by one when the
edge $v_1v_2$ is inserted to merge $T_1$ and $T_2$ into a single
component $T$. If $T_1$ and $T_2$ are both paths, then so is $T$ (see
Figure~\ref{fig:ext}, left). Otherwise, if $T_i$, for $i=1$, or $i=2$,
or both, is an odd cycle, then remove an edge $v_ix_i$, incident with
$v_i$ in $T_i$, to turn $T$ into a path (see Figure~\ref{fig:ext},
center and right). Finally, if the new path is odd and the two
endpoints are adjacent, insert the edge which connects its endpoints
to turn it into an odd cycle. Either way, we obtain a valid factor
$\Qo$ with one less component than $\Po$, which contradicts the
minimality of $\Po$.
\end{proof}

\begin{figure}[htb]
\centering \includegraphics[width=12cm]{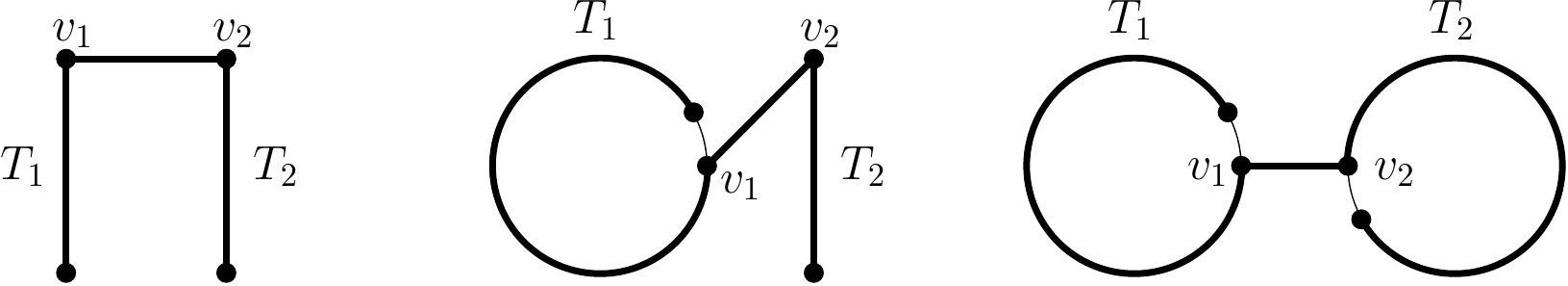} \caption{Three
instances of Claim~\ref{cl:edgeendpoint}} \label{fig:ext}
\end{figure}

\begin{claim}\label{cl:chord}
Each odd cycle of $\Po$ is chordless.
\end{claim}

\begin{proof}
Assume that an odd cycle $C$ of $\Po$ admits a chord. Since $C$ is odd
and $G$ is cubic, not every vertex on $C$ belongs to a chord.  We can
then select a chord $uv$ of $C$ and a vertex $w$, adjacent to $u$ on
$C$, such that $w$ does not belong to a chord of $C$.  Let $x$ be the
neighbor of $v$ on $C$ lying on the portion of $C$ between $v$ and $u$
which avoids $w$. We now replace $C$ in $\Po$ by the path $(C\cup uv)
\setminus \{xv,uw\}$ (see Figure~\ref{fig:chord}, left), As $w$
belongs to no chord of $C$, the endpoints $w$ and $x$ of this new (odd) path
are non-adjacent, so the obtained factor $\Qo$ is indeed valid,
contains the same number of components as $\Po$, and has less odd
cycles than $\Po$. This contradicts the minimality of $\Po$.
\end{proof}

\begin{figure}[htb]
\centering
\includegraphics[width=8cm]{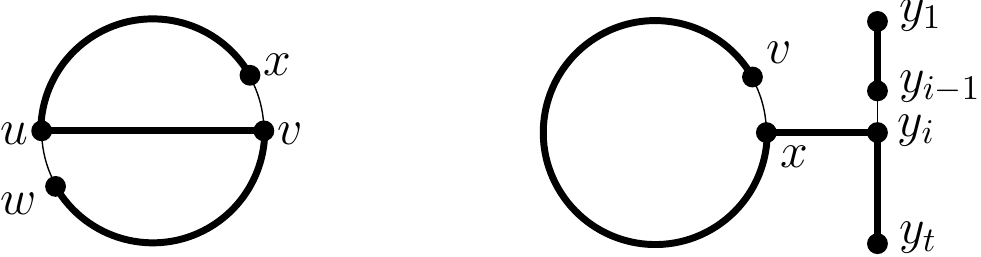}
\caption{Claim~\ref{cl:chord} (left) and Claim~\ref{cl:edgecyclepath2} (right).}\label{fig:chord}
\end{figure}

\begin{claim}\label{cl:edgecyclepath2}
Assume that $G$ contains an edge $xy_i$ such that $x$ is on an odd
cycle $C$ of $\Po$ and $y_i$ is a vertex of a path $P=y_1y_2\cdots
y_t$ of $\Po$. Then $i$ is even, $t$ is odd, and if $i\ne 2$ then
$y_1y_{i-1}$ is an edge of $G$ and if $i\ne t-1$ then $y_{i+1}y_t$
is an edge of $G$.
\end{claim}

\begin{proof}
Let $v$ be a neighbor of $x$ on $C$. If $t$ is even, or if $t$ and $i$
are both odd, then at least one of the paths $y_1y_2\cdots y_{i-1}$
and $y_{i+1}y_{i+2}\cdots y_t$ contains an even number of
vertices. Assume then that $i-1$ is even (otherwise reverse the
labeling $1,2,...,t$ on $P$). Let $\Qo$ be obtained from $\Po$ by
replacing $C$ and $P$ by (1) the path which is the union of
$C\setminus vx$, the edge $xy_i$, and the path $y_{i+1}y_{i+2}\cdots
y_t$ and (2) the even path $y_1y_2\cdots y_{i-1}$ (see
Figure~\ref{fig:chord}, right). By Claim~\ref{cl:edgeendpoint}, $vy_t$
is not an edge of $G$ and therefore $\Qo$ is indeed a valid
factor. $\Qo$ contains the same number of components as $\Po$ but has
one odd cycle less, a contradiction.

The contradiction above implies that $t$ is odd and $i$ is even.
So we assume now that this is indeed the case, but $i\ne 2$ and
$y_1$ and $y_{i-1}$ are non-adjacent. We now construct $\Qo$
exactly as we did before. This time the path $y_1y_2\cdots
y_{i-1}$ is odd, however, its endpoints $y_1$ and $y_{i-1}$ are
non-adjacent, and $i\ne 2$ guarantees that it is not an isolated
vertex. Consequently, $\Qo$ is a valid factor also in this case
and the same contradiction still holds. The case where $i=t-1$ and
$y_t$ is non-adjacent to $y_{i+1}$ is handled identically, after
reversing the labeling along $P$.
\end{proof}

\begin{claim}\label{cl:2edgescyclepath}
Every two vertices on an odd cycle of $\Po$ are connected to two
distinct odd paths of $\Po$.
\end{claim}

\begin{proof}
Let $C$ be an odd cycle of $\Po$. By Claim \ref{cl:chord}, $C$ is
chordless, so every vertex of $C$ is adjacent with one vertex on
another component of $\Po$. By Claims \ref{cl:edgeendpoint} and
\ref{cl:edgecyclepath2}, that other component is necessarily an
odd path. Assume now to the contrary, that $P=y_1y_2\cdots y_t$ is
a path of $\Po$ and that $G$ contains two edges $xy_i$ and $zy_j$,
with $i< j$. By Claim~\ref{cl:edgecyclepath2}, $i$ and $j$ are
even, and therefore $i\le j-2$. It follows that $j\ne 2$ and $i\ne
t-1$ and so by Claim~\ref{cl:edgecyclepath2}, $G$ contains the two
edges $y_1y_{j-1}$ and $y_{i+1}y_t$. Let $v$ be a neighbor of $x$
on $C$ (say $v\neq z$, although this is not really essential). Let
$\Qo$ be obtained from $\Po$ by replacing $C$ and $P$ by the path
$(C\cup
P\cup\{xy_i,y_1y_{j-1},y_{i+1}y_t\})\setminus\{xv,y_iy_{i+1},y_{j-1}y_j\}$
(see Figure~\ref{fig:cl45}, left). $C$ and $P$ are both odd, so
the new merged path is even and $\Qo$ is a valid factor with one
component less than $\Po$, a contradiction.
\end{proof}

\begin{figure}[htb]
\centering
\includegraphics[width=10cm]{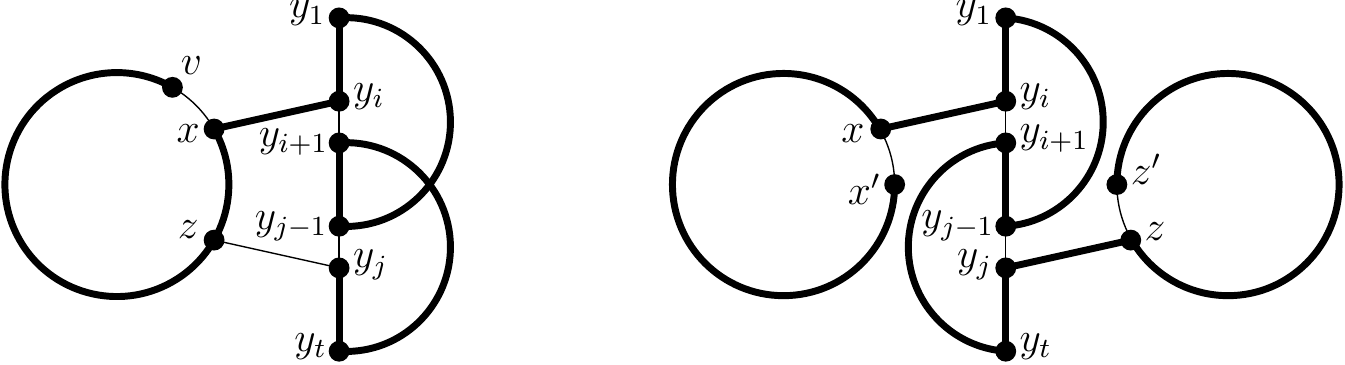}
\caption{Claim~\ref{cl:2edgescyclepath} (left) and
Claim~\ref{cl:cycle2path}
  (right).}\label{fig:cl45}
\end{figure}

\begin{claim}\label{cl:cycle2path}
A path of $\Po$ cannot be connected to two distinct odd cycles of
$\Po$.
\end{claim}

\begin{proof}
Assume that $\Po$ contains a path $P=y_1y_2\cdots y_t$ and two odd
cycles $C$ and $C'$, such that $G$ contains two edges $xy_i$ and
$zy_j$ with $i<j$, $x\in C$, and $z\in C'$. By
Claim~\ref{cl:edgecyclepath2}, $i$ and $j$ are even, and therefore
$i\le j-2$. It follows that $j\ne 2$ and $i\ne t-1$, so by
Claim~\ref{cl:edgecyclepath2}, $G$ contains the two edges
$y_1y_{j-1}$ and $y_{i+1}y_t$. Let $x'$ be a neighbor of $x$ on
$C$ and $z'$ a neighbor of $z$ on $C'$. Let $\Qo$ be obtained from
$\Po$ by removing $P$, $C$ and $C'$ and inserting the path
$$(C\cup C'\cup
P\cup\{xy_i,y_1y_{j-1},y_{i+1}y_t,y_iz\})\setminus\{xx',y_iy_{i+1},y_{j-1}y_j,zz'\},$$
see Figure~\ref{fig:cl45} (right). By Claim \ref{cl:edgeendpoint}
the endpoints of the new path, $x'$ and $z'$ are non-adjacent and
$\Qo$ is indeed a valid factor with two components less than
$\Po$, a contradiction.
\end{proof}

\begin{claim}
\label{p3c} The number of odd paths of $\Po$ is at least three
times the number of odd cycles of $\Po$.
\end{claim}

\begin{proof}
Claims \ref{cl:2edgescyclepath} and \ref{cl:cycle2path} establish
an injective function from the set of all vertices on odd cycles
of $\Po$, into the set of odd paths of $\Po$. Since there are at
least three vertices on each cycle, the assertion of the claim
immediately follows.
\end{proof}

\medskip

We are now set to complete the proof of
Theorem~\ref{th:main}.

\medskip

\noindent  Let $G$ be a cubic graph and let $\Po$ be the valid
factor defined above. As every component of $\Po$ is either a path
or a cycle, we now assign the colors $1$ and $2$, alternately to
vertices along every component and define the set $V_i$, $i=1,2$
as the set of all vertices with the color $i$. This coloring
schema is not uniquely defined and leaves room for further
refinement. Let us start with the odd cycles. On every odd cycle
$C$ an alternating coloring leaves two adjacent vertices $x$ and
$z$ with the same color, and $xz$ is the only edge of $C$ with the
same color on its two endvertices. We define $x$ to be the
\textbf{first} vertex of $C$ and $z$ the last one. Now let $x$ be
the first vertex of an odd cycle $C$ and let $P$ be the odd path,
to which $x$ is connected, according to Claim
\ref{cl:2edgescyclepath}, by an edge $xy$. We now assign a color
to $y$ such that $x$ and $y$ have distinct colors and then proceed
from $y$ to alternately color all vertices along $P$ in both
directions. Note that by Claim \ref{cl:cycle2path}, any path of $\Po$
is connected to at most one odd cycle of $\Po$, so the coloring above
is well-defined.
 For $(V_1,V_2)$ to be a bisection we should keep the
numbers of vertices colored $1$ and colored $2$ equal. The
alternating coloring schema guarantees this equality on every even
component. Among the vertices of an odd component, be it an odd
cycle, or an odd path, the color of the first (and the last)
vertex gains a majority of one. So far we determined the coloring
of all odd cycles and one odd path against every odd cycle. Let
$n$ be the number of odd cycles. The advantage of one color over
the other, so far, is at most $2n$. By Claim \ref{p3c} there are
still at least that many odd paths not yet colored. Each of them
can change the difference between $|V_1|$ and $|V_2|$ by one, in
any direction, by selecting the color of its first vertex. That,
plus the fact that the total number of odd components is even,
makes it possible to complete the coloring with $|V_1|=|V_2|$, as
required. An example of coloring obtained is depicted in
Figure~\ref{fig:proof}.

\smallskip

\begin{figure}[htb]
\centering \includegraphics[width=10cm]{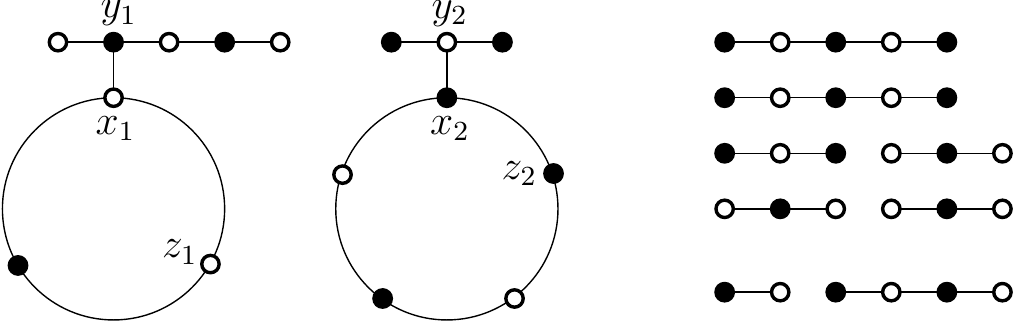} \caption{A typical
$2$-coloring of $\Po$.}\label{fig:proof}
\end{figure}

We now prove that there is no monochromatic component on more than
three vertices: Since no component of $\Po$ is an isolated vertex,
our alternating coloring schema provides, for every vertex, at
least one neighbor of the opposite color. That observation
eliminates monochromatic $K_{1,3}$. It remains to consider
monochromatic copies of $P_4$ (the path on $4$ vertices) and $K_3$
(the triangle). We say
that a vertex $x$ is \textbf{good}, if at least two of its three
neighbors differ from $x$ in color; $x$ is \textbf{bad} if it
shares his color with two of its neighbors. A monochromatic $P_4$ or $K_3$
contains two adjacent bad vertices of the same color. So, we only
need to show that no two such vertices are adjacent. Every inner
(not an endpoint) vertex of a path is necessarily good. The only
bad vertices on cycles may be the last vertices of the odd cycles.
Notice that the first vertex $x$ of an odd cycle $C$ differs in
color from one of its neighbors on $C$ (not the last one), and
from its neighbor $y$ on an odd path $P$, so $x$ is a good vertex.
It turns out then, that all bad vertices, if there are any, are
\textbf{external} (as defined in Claim \ref{cl:edgeendpoint}).
Therefore, by Claim \ref{cl:edgeendpoint}, no two bad vertices
from two distinct components of $\Po$ are adjacent. By the
definition of a valid factor, the two endpoints of an odd path are
non-adjacent; An odd cycle has at most one bad vertex (its last
vertex) and the two endpoints of an even path are of distinct
colors. In conclusion, there is no monochromatic $P_4$ or $K_3$, so $(V_1,V_2)$ is indeed a $5$-weak
bisection. \hfill $\square$

\medskip

As already mentioned, Theorem \ref{th:main} can be looked at, as a
(much) weaker version of the assertion of Tutte's 5-Flow
Conjecture. Not precisely weaker, because Theorem \ref{th:main}
holds for every cubic graph, while the 5-Flow Conjecture relates
only to bridgeless graphs. When searching potential directions to
deal with nowhere-zero flow problems, it is worth noting that we
can suggest a very short and simple proof of Theorem
\ref{th:main}, when restricted to cubic graphs which admits
perfect matchings, bridgeless graphs in particular. A similar
proof appears in~\cite{steff2}, although the result is not
explicitly stated there.

\vspace{10pt}

\begin{proof} Let $G=(V,E)$, be a cubic graph and $M$
 a perfect matching in $G$. Then $T=E \setminus M$ is a
\textbf{2-factor} (a union of disjoint simple cycles). Let $S$ be
a set of edges, which consists of two consecutive edges from each
odd cycle of $T$. Accordingly, $T'=T \setminus S$ is a union of
disjoint even cycles and even (odd number of edges) paths and, as
such, can be partitioned into two matchings $T_1$ and $T_2$. We
make sure to include in $T_1$ the two end-edges of each path  of
$T'$. That way, $S$ and $T_2$ have no common vertex, so every
connected component of $S \cup T_2$ is either a single edge of
$T_2$, or two edges of $S$. As the union of two matchings,
$P=M\cup T_1$ is bipartite. Let $(V_1,V_2)$ be any bipartition of
$P$. Since $M$ is a perfect matching, $(V_1,V_2)$ is a bisection.
An edge of $P$ connects two vertices of distinct colors, so all
monochromatic edges belong to $E \setminus P=S \cup T_2$.
Therefore there are at most two edges in a monochromatic component
and $(V_1,V_2)$ is indeed a $5$-weak bisection.
\end{proof}

Let us add that, following this proof, the number of monochromatic
paths on 3 vertices, can be bounded by the \textbf{oddness}
(minimum number of odd components in a 2-factor) of $G$.

\end{document}